\documentclass[12pt]{amsart}

\usepackage{amssymb}
\usepackage[margin=1.5in]{geometry}

\usepackage{amsthm,amsmath}
\usepackage{graphicx}
\usepackage{amscd}
\usepackage{latexsym}
\usepackage{mathrsfs}
\usepackage{xcolor}
\usepackage{enumerate}

\usepackage{soul}

\usepackage{tikz}
\usepackage{tikz}
\usetikzlibrary{arrows,snakes,backgrounds,patterns}

\makeatletter
\def\moverlay{\mathpalette\mov@rlay}
\def\mov@rlay#1#2{\leavevmode\vtop{%
\baselineskip\z@skip \lineskiplimit-\maxdimen
\ialign{\hfil$\m@th#1##$\hfil\cr#2\crcr}}}
\newcommand{\charfusion}[3][\mathord]{
#1{\ifx#1\mathop\vphantom{#2}\fi
     \mathpalette\mov@rlay{#2\cr#3}}
     \ifx#1\mathop\expandafter\displaylimits\fi}
\makeatother

\newcommand\abs[1]{\lvert#1\rvert}

\theoremstyle{plain}
\newtheorem{theorem}{Theorem}
\newtheorem{lemma}[theorem]{Lemma}

\theoremstyle{definition}

\newtheorem{example}[theorem]{Example}
\newtheorem{conjecture}[theorem]{Conjecture}

\theoremstyle{remark}

\newcommand{\Sch}{Schr\"{o}der }

\begin{document}

\title{Enumeration of Fuss-Schr\"{o}der paths}

\author{Suhyung An}
\address{Department of Mathematics\\
         Yonsei University\\
         Seoul 120-749, Republic of Korea}
\email{hera1973@yonsei.ac.kr}

\author{JiYoon Jung}
\address{Department of Mathematics\\
         Marshall University\\
         Huntington, WV 25755}
\email{jungj@marshall.edu}

\author{Sangwook Kim}
\address{Department of Mathematics\\
         Chonnam National University\\
         Gwangju 500-757, Republic of Korea}
\email{swkim.math@chonnam.ac.kr}

\begin{abstract}
   In this paper 
   we enumerate the number of $(k, r)$-Fuss-Schr\"{o}der paths of type $\lambda$.
   Y. Park and S. Kim studied small Schr\"{o}der paths with type~$\lambda$.
   Generalizing the results to small $(k, r)$-Fuss-Schr\"{o}der paths with type $\lambda$,
   we give a combinatorial interpretation for 
   the number of small $(k, r)$-Fuss-Schr\"{o}der paths of type $\lambda$ 
   by using Chung-Feller style.
   We also give 
   two sets of sparse noncrossing partitions of $[2(k + 1)n + 1]$ and $[2(k + 1)n + 2]$
   which are in bijection with the set of 
   all small and large, respectively, $(k, r)$-Fuss-Schr\"{o}der paths of type $\lambda$.
\end{abstract}


\maketitle


\section{Introduction}
\label{sec-introduction}

   A Dyck path of length $n$ is a lattice path from $(0,0)$ to $(n,n)$
   using east steps $E = (1,0)$ and north steps $N = (0,1)$ such that 
   it stays weakly above the diagonal line $y = x$.
   It is well-known that the number of all Dyck paths of length $n$
   is given by the famous Catalan numbers
   $$
   \frac{1}{n+1} \binom{2n}{n}.
   $$
   
   A large \Sch path of length $n$ is a lattice path from $(0,0)$ to $(n,n)$
   using east steps~$E$, north steps $N$, and diagonal steps $D = (1,1)$
   staying weakly above the diagonal line $y = x$.
   The number of all large \Sch paths of length $n$ is 
   $$
   \frac{1}{n} \sum_{k=1}^n \binom{n}{k-1} \binom{n}{k} 2^k.
   $$
   A small \Sch path of length $n$ is a large \Sch path with no diagonal 
   steps on the diagonal line.
   The number of all small \Sch paths of length~$n$ is the half of the number of
   large \Sch paths of length $n$.
   Note that a Dyck path is a large \Sch path without using diagonal steps.
   
   For a large \Sch path (and hence for a Dyck path), its \emph{type}
   is the integer partition formed by the length of the maximal adjacent 
   east steps.
   For example, the large \Sch path $NENNNEEDNEEDNE$ has type 
   $\lambda = (2, 2, 1, 1)$.
   The enumeration of Dyck paths by type is done since 
   Kreweras~\cite{Kreweras} introduced noncrossing partitions 
   while the enumeration of large \Sch paths by type is recently done by 
   An, Eu, and Kim~\cite{AnEuKim}.
   The number of small \Sch paths of given type is \emph{not} the 
   half of the number of large \Sch paths and it is enumerated by 
   Park and Kim~\cite{ParkKim}.
   
   Now we introduce Fuss analogue of Dyck and \Sch paths.
   Given a positive number~$k$, a $k$-Fuss-Catalan path of length $n$ is
   a path from $(0,0)$ to $(n,kn)$ using east steps~$E$ and north steps $N$
   such that it stays weakly above the line $y = kx$.
   The number of all $k$-Fuss-Catalan paths of length $n$ is given by
   the Fuss-Catalan numbers 
   $$
   \frac{1}{kn+1} \binom{(k+1)n}{n}
   $$
   and Armstrong~\cite{Armstrong} enumerates the number of 
   $k$-Fuss-Catalan paths of given type.
   
   For $k, r$ $(1 \le r \le k)$, a large $(k, r)$-Fuss-\Sch path of length $n$ is 
   a path~$\pi$ from $(0,0)$ to $(n, kn)$ using east steps, north steps, and 
   diagonal steps that satisfies the following two conditions:
   \begin{enumerate}[(C1)]
      \item
      the path $\pi$ never passes below the line $y =kx$, and 
      \item
      the diagonal steps of $\pi$ are only allowed to go from 
      the line $y = kj + r - 1$ to the line $y =kj +r$, for some $j$.
   \end{enumerate}
   A small Fuss-\Sch path is a large Fuss-\Sch path with no diagonal steps touching 
   the line $y =kx$.
   The number of small $(k, r)$-Fuss-\Sch paths with fixed length and number 
   of diagonal steps is independent of $r$ and it is given 
   by Eu and Fu~\cite{EuFu}.
   
   We will provide the number of small $(k,r)$-Fuss-\Sch
   paths of given length and type. 
   We also give two conjectures about Fuss-\Sch paths and sparse noncrossing partitions
   which might be useful for the formula for the number of large
   $(k,r)$-Fuss-\Sch paths of given type and length.
   
\section{Dyck and \Sch paths by type}
\label{sec-previous-results}

   In this section, we introduce previous results about the numbers of Dyck and 
   \Sch paths with given length and type.
   
   Given an integer partition $\lambda$, 
   we set $m_\lambda := m_1(\lambda)! m_2(\lambda)! m_3(\lambda)! \cdots$, 
   where $m_i(\lambda)$ is the number of parts of $\lambda$ equal to $i$. 
   Note that $m_\lambda$ here is not the monomial symmetric function.
   We use $\abs{\lambda}$ for the sum of the parts of $\lambda$.

   First, we begin with the number of Dyck paths of given type.   
   Kreweras~\cite{Kreweras} shows the following theorem using recursions and 
   Liaw et al.~\cite{Liaw} give a bijective proof.
   
   \begin{theorem}
      The number of 
      \begin{enumerate}
         \item
         Dyck paths of length $n$ with type $\lambda = (\lambda_1, \dots, \lambda_\ell)$
         \item
         noncrossing partitions of $[n]$ with type $\lambda$
      \end{enumerate}
      is
      $$
      \frac{n (n-1) \cdots (n -\ell(\lambda) + 2)}{m_\lambda}.
      $$
   \label{dyck}
   \end{theorem}
   
   If diagonal steps are allowed, Theorem~\ref{dyck} is generalized to \Sch path cases.
   An, Eu, and Kim~\cite{AnEuKim} enumerate the number of large \Sch paths
   of given length and type.
   
   \begin{theorem}
      The number of 
      \begin{enumerate}
         \item
         large \Sch paths of length $n$ with type 
         $\lambda = (\lambda_1, \dots, \lambda_\ell)$ 
         \item
         sparse noncrossing set partitions of $[n + \abs{\lambda} + 1]$ with arc type $\lambda$
      \end{enumerate}
      is
      $$
      \frac{1}{\abs{\lambda}+1} \binom{n}{\abs{\lambda}} \binom{n+1}{\ell}
      \frac{\ell !}{m_\lambda}.
      $$
   \end{theorem}

   It is well-known that the number of small \Sch paths of length $n$ is the half of
   the number of large \Sch paths of length $n$.
   This is not the case when we count the number of small \Sch paths of fixed type.
   Park and Kim~\cite{ParkKim} provide the number of small \Sch paths of given length and type.
   
   \begin{theorem}
      The number of 
      \begin{enumerate}
         \item
         small \Sch paths of length $n$ with type $\lambda = (\lambda_1, \dots, \lambda_\ell)$
         \item
         large \Sch paths of length $n$ with type $\lambda$ 
         with no diagonal steps after the last up step
         \item
         connected sparse noncrossing set partitions of $[n + \abs{\lambda} + 1]$ 
         with arc type $\lambda$
      \end{enumerate}
      is given by
      $$
      \frac{1}{n+1} \binom{n-1}{\abs{\lambda}-1}\binom{n+1}{\ell}\frac{\ell!}{m_\lambda}.
      $$
   \end{theorem}

\section{Fuss-\Sch paths}
\label{sec-Fuss-Sch-paths}

   In this section, 
   we consider Fuss analogue of Dyck and \Sch paths of given length and type,
   i.e. Fuss-Catalan paths and Fuss-\Sch paths of fixed length $n$ with type $\lambda$.
   Our goal is 
   to enumerate the number of small Fuss-\Sch paths of length $n$ with type $\lambda$.
   
   First, we begin with the number of Fuss-Catalan paths of given type.
   
   \begin{theorem}
   [\cite{Armstrong}]
      The number of $k$-Fuss-Catalan paths of type $\lambda$ is
      $$
      \frac{(kn)!}{m_\lambda \cdot (kn + 1 - \ell (\lambda))!}.
      $$
   \end{theorem}

   The following shows that the number of large $(k,r)$-Fuss-\Sch paths of given length and type
   is independent of $r$.
    
   \begin{lemma}
   \label{r-independent}
      Let $\mathcal{A}_{n, \lambda}^{(k,r)}$ be the set of 
      all large $(k, r)$-Fuss-\Sch paths of length~$n$ with type $\lambda$.
      Then there is a bijection between $\mathcal{A}_{n, \lambda}^{(k,i)}$ and
      $\mathcal{A}_{n, \lambda}^{(k,j)}$ for $1 \le i < j \le k$.
   \end{lemma}
   
   \begin{proof}
      Let $\pi$ be a path in $\mathcal{A}_{n, \lambda}^{(k,i)}$.
      If $\pi$ has a diagonal step $D$ from the level $kt + i - 1$ to the level $kt + i$, 
      then it can be decomposed into $\mu_1 D \mu_2 \omega N \mu_3$, 
      where $\omega$ is the section of east steps on level $kt + j - 1$, 
      $N$ is the north step from the level $kt + j - 1$ to $kt + j$,
      $\mu_1$ goes from $(0,0)$ to the level $kt + i - 1$,
      $\mu_2$ goes from the level $kt + i$ to the level $kt + j - 1$, and
      $\mu_3$ goes from the level $kt + j$ to $(n, kn)$.
      Let $\tau$ be the path $\mu_1 N \mu_2 \omega D \mu_3$.
      Then $\tau$ has a diagonal step from the level $kt + j - 1$ to the level $kt + j$ and 
      its type is the same as the type of $\pi$.
      By applying the similar operations to all diagonal steps in $\pi$,
      we get a path in $\mathcal{A}_{n, \lambda}^{(k,j)}$.
   \end{proof}
   
   Similarly, one can show that the number of small $(k,r)$-Fuss-\Sch paths of given
   length and type is independent of $r$.
   The following theorem provides the number of small Fuss-\Sch paths of fixed length and type.   
   
   \begin{theorem}
   \label{th-small-rule}
      The number of small $(k, r)$-Fuss-\Sch paths of type $\lambda$ ($1 \le r \le k$) is
      $$
      \binom{n-1}{\abs{\lambda} - 1}\binom{nk}{\ell - 1} \frac{(\ell - 1)!}{m_\lambda}
      = \frac{1}{nk+1} 
      \binom{n-1}{\abs{\lambda} - 1}\binom{nk + 1}{\ell} \frac{\ell !}{m_\lambda}.
      $$
      When $k = 1$, this becomes
      $$
      \binom{n-1}{\abs{\lambda} - 1}\binom{n}{\ell - 1} \frac{(\ell - 1)!}{m_\lambda}
      = \frac{1}{n+1} 
      \binom{n-1}{\abs{\lambda}-1}\binom{n+1}{\ell}\frac{\ell!}{m_\lambda}.
      $$
      When $\abs{\lambda} = n$, this becomes
      $$
      \binom{nk}{\ell - 1} \frac{(\ell - 1)!}{m_\lambda}
      = \frac{1}{nk+1} 
      \binom{nk+1}{\ell} \frac{\ell!}{m_\lambda}.
      $$
   \label{FSpath}
   \end{theorem}   
   
   Note that 
   the number of small $(k, r)$-Fuss-\Sch paths of the case $k = 1$ 
   is the number of small \Sch paths    
   since paths staying above the line $y=x$ are considered.
   In the case of $\abs{\lambda} = n$, 
   the number of small $(k, r)$-Fuss-\Sch paths is the number of $k$-Fuss-Catalan paths
   since we count paths without using diagonal steps only.
    
   \begin{proof}[Proof of Theorem~\ref{FSpath}]
      The number of small $(k, r)$-Fuss-\Sch paths 
      of type~$\lambda$ ($1 \le r \le k$) is independent of $r$ (by Lemma~\ref{r-independent}), and 
      $r=k$ is assumed in this proof.      
      
      To show the number of small $(k, k)$-Fuss-\Sch paths 
      of type $\lambda = \lambda_1 \lambda_2 \ldots \lambda_\ell$ is
      $\frac{1}{nk+1} 
      \binom{n-1}{\abs{\lambda} - 1}\binom{nk + 1}{\ell} \frac{\ell!}{m_\lambda}$,
      we first consider all the paths from $(0,0)$ to $(n,kn)$ of type $\lambda$
      using east steps, north steps, and diagonal steps such that 
      the diagonal steps are only allowed to go in the $ik$th rows for $2 \leq i \leq n$.
      Being such a path,  
      $n-\abs{\lambda}$ rows are chosen from $2k, 3k, \ldots, nk$th
      rows for $n-\abs{\lambda}$ diagonal steps, and
      $\ell$ lines are needed from $nk+1$ horizontal lines 
      for $\ell$ east runs (i.e. maximal consecutive east steps).
      Since the $\ell$ east runs are ordered in $\frac{\ell!}{m_\lambda}$ different ways,
      there are 
      total $\binom{n-1}{\abs{\lambda} - 1}\binom{nk + 1}{\ell} \frac{\ell!}{m_\lambda}$ paths
      satisfying the above conditions. 
      
      We say a diagonal step of the paths has a flaw if it is located below the line $y=kx+k$, 
      and a north step has a flaw if it is below the line $y=kx$.
      North steps in $(ik-1)$th, $(ik-2)$th, $\ldots$, $(ik-k+1)$th rows 
      between $y=kx$ and $y=kx+k$ are also considered as flawed steps
      if a flawed diagonal step is contained in the $ik$th row. 
      Note that each path has $j$ flaws for some $j \in [1, kn]$
      if and only if it is not a small $(k, k)$-Fuss-\Sch path.
      The following rules identify $nk$ flawed paths 
      corresponding to a small $(k, k)$-Fuss-\Sch path 
      by increasing the number of flaws from the small $(k, k)$-Fuss-\Sch path one by one,
      and vice versa.
            
      \begin{enumerate}
         \item
         Increasing one flaw:    
         \begin{enumerate}
            \item[(1a)]
            See a given path as a sequence on $\{E, N, D\}^{(k+1)n}$, and  
            take the leftmost east run which is located
            right before D, NE, any flawed step, or nothing 
            (i.e. the east run is the last steps). 
            Move the east run to the left on the sequence 
            until the path get exactly one more flaw
            on condition that the east run does not pass D, NE, or any flawed step 
            without an increase on the number of flaws.
            Example~\ref{eg-small-rule-1} is given for this basic case.
            \item[(1b)]
            If the leftmost east run passes D, NE, or any flawed step
            without an increase on the number of flaws as seen in Example~\ref{eg-small-rule-2}, 
            stop the moving right after passing the step, 
            and do (1a) for the following east runs again.  
            \item[(1c)]
            In the case that 
            the leftmost east run in (1a) contains the first east step of the sequence,
            apply $(-1,-k)$-circular shifts 
            (i.e. shifting each step of the path 
            along a vector $(-1~\text{mod}~n, -k~\text{mod}~kn)$ to get another path) 
            repeatedly until obtaining a new path from $(0,0)$ to $(n, kn)$ such that 
            the first step is not an east step and the $k$th row doesn't have a diagonal step.
            If there is no such path, the original path already has $kn$ flaws and we may stop.
            See Example~\ref{eg-small-rule-3}.
            To keep type $\lambda$ for the new path, 
            we consider two east runs of the original path are separately
            even if they are connected after circular shifts.              
            \item[(1d)]
            If a new path is obtained in (1c),
            take the leftmost one (say the $s$th east run) and 
            the the rightmost one (say the $(s+t)$th east run) among east runs located
            right before D, NE, any flawed step, or on the same line with another east run.  
            Move the $(s+t-u)$th east run
            to the position of the $(s+t-u-1)$th east run ($0 \leq u \leq t-1$)
            and, after that, apply (1a) to the $s$th east run.
            There are two examples for the circular shift in Example~\ref{eg-small-rule-4}.
         \end{enumerate}
         \item
         Decreasing one flaw: 
         \begin{enumerate}
            \item[(2a)]
            Note that there is always an east run 
            right before the leftmost flawed step on a sequence.
            Move the east run to the right 
            until it meets D, E, NE, or the second leftmost flawed step if those steps exist.
            Otherwise, the east run becomes the last steps of the sequence.
            Look at the second part in Example~\ref{eg-small-rule-1}.
            \item[(2b)]
            If another east run is located right before flawless D or flawless NE
            which is prior the east run used in (2a) 
            as Figure~\ref{small-rule-2}(b) in Example~\ref{eg-small-rule-2},
            move the another one to the right
            until meeting with next D, NE or a flawed step.
            \item[(2c)]
            If two east runs share the same horizontal line
            (i.e. east lines are back to back on a sequence) 
            while any east run rearrangement is done,
            move the right east run to the right
            until it meets a next east run or a horizontal line $y=(w-1)k$
            where $wk$th row is the highest row containing no diagonal step
            among $ik$th rows ($2 \leq i \leq n$)
            as Example~\ref{eg-small-rule-4}. 
            \item[(2d)]
            If (2c) has ever been applied 
            like both examples in Example~\ref{eg-small-rule-4}, 
            then we need a $(-w+1~\text{mod}~n, (-w+1)k~\text{mode}~kn)$-circular shift.
            Otherwise, we may stop.
         \end{enumerate}      
      \end{enumerate}      
      Therefore, there are 
      $\frac{1}{nk+1} \binom{n-1}{\abs{\lambda} - 1}\binom{nk + 1}{\ell} \frac{\ell!}{m_\lambda}$
      small $(k, k)$-Fuss-\Sch paths of type $\lambda$.      
   \end{proof}

   In the next 4 examples, 
   we consider paths from $(0,0)$ to $(n,kn)$ of type $\lambda$
   using east steps, north steps, and diagonal steps such that 
   the diagonal steps are only allowed to go in the $ik$th rows for $2 \leq i \leq n$
   where $n=4$, $k=2$, and $\lambda = (2,1)$.   
   
   \begin{example}
   \label{eg-small-rule-1}
      The $2$-east run in Figure~\ref{small-rule-1}(a) is the leftmost one 
      among east runs located right before D, NE, any flaw step, or nothing.
      Move the $2$-east run to the left one unit on the sequence
      so that the path get one more flaw, a diagonal step below the line $y=2x+2$, 
      as in Figure~\ref{small-rule-1}(b). 
      
      In the reverse, 
      the leftmost (and unique) flawed step is a diagonal step, and 
      there is a $2$-east run right before the flawed step
      in Figure~\ref{small-rule-1}(b).
      Since there is no D, E, NE, or the second leftmost flawed step after the $2$-east run,
      move the $2$-east run to the right until it becomes the last steps of the sequence
      as in Figure~\ref{small-rule-1}(a).   
         
   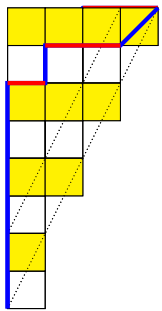
\begin{figure}[h]
   \begin{center}
   \begin{tikzpicture}[scale=0.5]
      \hspace{-2cm}
      \draw (0,11) rectangle (1,12);
      \filldraw[fill=yellow] (0,12) rectangle (1,13);
      \draw (0,13) rectangle (1,14);
      \filldraw[fill=yellow] (0,14) rectangle (1,15);
      \draw (0,15) rectangle (1,16);
      \filldraw[fill=yellow] (0,16) rectangle (1,17);
      \draw (0,17) rectangle (1,18);
      \filldraw[fill=yellow] (0,18) rectangle (1,19);
      \filldraw[fill=yellow] (1,14) rectangle (2,15);
      \draw (1,15) rectangle (2,16);
      \filldraw[fill=yellow] (1,16) rectangle (2,17);
      \draw (1,17) rectangle (2,18);
      \filldraw[fill=yellow] (1,18) rectangle (2,19);
      \filldraw[fill=yellow] (2,16) rectangle (3,17);
      \draw (2,17) rectangle (3,18);
      \filldraw[fill=yellow] (2,18) rectangle (3,19);
      \filldraw[fill=yellow] (3,18) rectangle (4,19);
      \draw[densely dotted] (0,13)--(3,19);
      \draw[densely dotted] (0,11)--(4,19);
      \draw[ultra thick,blue] (0,11)--(0,17)--(1,17)--(1,18)--(2,19)--(4,19);
      \draw[ultra thick,red] (0,17)--(1,17);
      \draw[ultra thick,red] (2,19)--(4,19);
      \draw (1,9) node{(a) No flaws};
      
      \hspace{5cm}
      \draw (0,11) rectangle (1,12);
      \filldraw[fill=yellow] (0,12) rectangle (1,13);
      \draw (0,13) rectangle (1,14);
      \filldraw[fill=yellow] (0,14) rectangle (1,15);
      \draw (0,15) rectangle (1,16);
      \filldraw[fill=yellow] (0,16) rectangle (1,17);
      \draw (0,17) rectangle (1,18);
      \filldraw[fill=yellow] (0,18) rectangle (1,19);
      \filldraw[fill=yellow] (1,14) rectangle (2,15);
      \draw (1,15) rectangle (2,16);
      \filldraw[fill=yellow] (1,16) rectangle (2,17);
      \draw (1,17) rectangle (2,18);
      \filldraw[fill=yellow] (1,18) rectangle (2,19);
      \filldraw[fill=yellow] (2,16) rectangle (3,17);
      \draw (2,17) rectangle (3,18);
      \filldraw[fill=yellow] (2,18) rectangle (3,19);
      \filldraw[fill=yellow] (3,18) rectangle (4,19);
      \draw[densely dotted] (0,13)--(3,19);
      \draw[densely dotted] (0,11)--(4,19);
      \draw[ultra thick,blue] (0,11)--(0,17)--(1,17)--(1,18)--(3,18)--(4,19);
      \draw[ultra thick,red] (0,17)--(1,17);
      \draw[ultra thick,red] (1,18)--(3,18);
      \draw (1,9) node{(b) 1 flaw};
   \end{tikzpicture}
   \end{center}
   \vspace{-0.2cm}
   \caption{
   Rules for adding and subtracting flaws
   in Theorem~\ref{th-small-rule}(1a) and \ref{th-small-rule}(2a).}
   \label{small-rule-1}
   \end{figure} 
   \end{example}
   
   \begin{example}
   \label{eg-small-rule-2}
      The $1$-east run in Figure~\ref{small-rule-2}(a) is the leftmost one 
      among east runs right before D, NE, any flaw step, or nothing.
      Moving the $1$-east run to the left, it passes D without increase on the number of flaws.
      Hence, stop the moving right after passing the diagonal step, and 
      move the second leftmost one, the $2$-east run, to the left one unit on the sequence
      so that the path get one more flaw, a north step below the line $y=2x$, 
      as in Figure~\ref{small-rule-2}(b). 
   
      Inversely,
      the leftmost flawed step is the highest north step in Figure~\ref{small-rule-2}(b), and 
      move a $2$-east run right before the flawed step
      until it becomes the last steps of the sequence.
      After that, 
      move the $1$-east run, right before flawless D prior the $2$-east run,
      to the right $2$ units until meeting with NE
      as in Figure~\ref{small-rule-2}(a).
         
   \begin{figure}[h]
   \begin{center}
   \begin{tikzpicture}[scale=0.5]
      \hspace{-2cm}
      \draw (0,11) rectangle (1,12);
      \filldraw[fill=yellow] (0,12) rectangle (1,13);
      \draw (0,13) rectangle (1,14);
      \filldraw[fill=yellow] (0,14) rectangle (1,15);
      \draw (0,15) rectangle (1,16);
      \filldraw[fill=yellow] (0,16) rectangle (1,17);
      \draw (0,17) rectangle (1,18);
      \filldraw[fill=yellow] (0,18) rectangle (1,19);
      \filldraw[fill=yellow] (1,14) rectangle (2,15);
      \draw (1,15) rectangle (2,16);
      \filldraw[fill=yellow] (1,16) rectangle (2,17);
      \draw (1,17) rectangle (2,18);
      \filldraw[fill=yellow] (1,18) rectangle (2,19);
      \filldraw[fill=yellow] (2,16) rectangle (3,17);
      \draw (2,17) rectangle (3,18);
      \filldraw[fill=yellow] (2,18) rectangle (3,19);
      \filldraw[fill=yellow] (3,18) rectangle (4,19);      
      \draw[densely dotted] (0,13)--(3,19);
      \draw[densely dotted] (0,11)--(4,19);
      \draw[ultra thick,blue] (0,11)--(0,16)--(1,17)--(1,18)--(2,18)--(2,19)--(4,19);
      \draw[ultra thick,red] (1,18)--(2,18);
      \draw[ultra thick,red] (2,19)--(4,19);
      \draw (1,9) node{(a) No flaws};
      
      \hspace{5cm}
      \draw (0,11) rectangle (1,12);
      \filldraw[fill=yellow] (0,12) rectangle (1,13);
      \draw (0,13) rectangle (1,14);
      \filldraw[fill=yellow] (0,14) rectangle (1,15);
      \draw (0,15) rectangle (1,16);
      \filldraw[fill=yellow] (0,16) rectangle (1,17);
      \draw (0,17) rectangle (1,18);
      \filldraw[fill=yellow] (0,18) rectangle (1,19);
      \filldraw[fill=yellow] (1,14) rectangle (2,15);
      \draw (1,15) rectangle (2,16);
      \filldraw[fill=yellow] (1,16) rectangle (2,17);
      \draw (1,17) rectangle (2,18);
      \filldraw[fill=yellow] (1,18) rectangle (2,19);
      \filldraw[fill=yellow] (2,16) rectangle (3,17);
      \draw (2,17) rectangle (3,18);
      \filldraw[fill=yellow] (2,18) rectangle (3,19);
      \filldraw[fill=yellow] (3,18) rectangle (4,19);
      \draw[densely dotted] (0,13)--(3,19);
      \draw[densely dotted] (0,11)--(4,19);
      \draw[ultra thick,blue] (0,11)--(0,16)--(1,16)--(2,17)--(2,18)--(4,18)--(4,19);
      \draw[ultra thick,red] (0,16)--(1,16);
      \draw[ultra thick,red] (2,18)--(4,18);
      \draw (1,9) node{(b) 1 flaw};   
   \end{tikzpicture}
   \end{center}
   \vspace{-0.2cm}
   \caption{
   Rules for adding and subtracting flaws
   in Theorem~\ref{th-small-rule}(1b) and \ref{th-small-rule}(2b).}
   \label{small-rule-2}
   \end{figure}
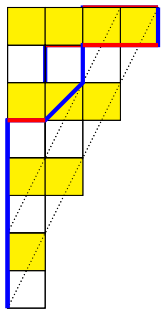 
   \end{example}
   
   \begin{example}
   \label{eg-small-rule-3}
      If $(-1,-2)$-circular shift is applied to the path in Figure~\ref{small-rule-3} only once,
      we get a path containing a diagonal step  on the $2$nd row. 
      If $(-1,-2)$-circular shift is applied twice,
      a path having $2$-east run as first steps is obtained. 
      Three $(-1,-2)$-circular shifts generate 
      a disconnected path not from $(0,0)$ to $(4, 8)$.
      Hence, one more flaw cannot be added, and
      it is natural since 
      the path in Figure~\ref{small-rule-3} is already containing $8$ flaws fully.
   
   \begin{figure}[h]
   \begin{center}
   \begin{tikzpicture}[scale=0.5]
      \hspace{0cm}
      \draw (0,11) rectangle (1,12);
      \filldraw[fill=yellow] (0,12) rectangle (1,13);
      \draw (0,13) rectangle (1,14);
      \filldraw[fill=yellow] (0,14) rectangle (1,15);
      \draw (0,15) rectangle (1,16);
      \filldraw[fill=yellow] (0,16) rectangle (1,17);
      \draw (0,17) rectangle (1,18);
      \filldraw[fill=yellow] (0,18) rectangle (1,19);
      \filldraw[fill=yellow] (1,14) rectangle (2,15);
      \draw (1,15) rectangle (2,16);
      \filldraw[fill=yellow] (1,16) rectangle (2,17);
      \draw (1,17) rectangle (2,18);
      \filldraw[fill=yellow] (1,18) rectangle (2,19);
      \filldraw[fill=yellow] (2,16) rectangle (3,17);
      \draw (2,17) rectangle (3,18);
      \filldraw[fill=yellow] (2,18) rectangle (3,19);
      \filldraw[fill=yellow] (3,18) rectangle (4,19);
      \draw[densely dotted] (0,13)--(3,19);
      \draw[densely dotted] (0,11)--(4,19);
      \draw[ultra thick,blue] (0,11)--(1,11)--(1,14)--(2,15)--(4,15)--(4,19);
      \draw[ultra thick,red] (0,11)--(1,11);
      \draw[ultra thick,red] (2,15)--(4,15);
      \draw (1,9) node{8 flaws};
   \end{tikzpicture}
   \end{center}
   \vspace{-0.2cm}
   \caption{
   Rules for adding and subtracting flaws
   in Theorem~\ref{th-small-rule}(1c).}
   \label{small-rule-3}
   \end{figure} 
   \end{example}
   
   \begin{example}
   \label{eg-small-rule-4}   
      \hspace{0cm}
      \begin{enumerate}      
         \item
         If $(-1,-2)$-circular shift is applied 
         to the path in Figure~\ref{small-rule-4}(a) twice,
         we obtain a new path from $(0,0)$ to $(4, 8)$ such that 
         the first step is not an east step and the $2$nd row doesn't have a diagonal step
         as in Figure~\ref{small-rule-4}(b).
         To keep type $\lambda$ for the new path,
         $3$ east steps on the same horizontal line after circular shifts are considered as 
         two east runs, $2$-east run and $1$-east run in order, separately.
         Now in Figure~\ref{small-rule-4}(b), 
         take the leftmost one ($2$-east run) and 
         the rightmost one ($1$-east run) among east runs 
         right before D, NE, any flawed step, or on the same line with another east run.  
         The $1$-east run is already 
         in the (next) position of the $2$-east run on the sequence, and
         the $2$-east run is moved to the left one unit 
         so that a path gets one more flawed north step 
         as shown in Figure~\ref{small-rule-4}(c).
   
         To the opposite direction,
         the $2$-east run right before the leftmost flawed step 
         in Figure~\ref{small-rule-4}(c)
         is moved to the right one unit on the sequence and meets another east run
         as in Figure~\ref{small-rule-4}(b). 
         Since two east runs share the same horizontal line, and
         $6$th ($w=3$) row is the highest row containing no diagonal step 
         among $ik$th rows ($2 \leq i \leq n$), and
         the right east run ($1$-east run) is already on the horizontal line $y=4$,
         the only necessary thing to get a path in Figure~\ref{small-rule-4}(a) is 
         a $(2, 4)$-circular shift.
      
   \begin{figure}[h]
   \begin{center}
   \begin{tikzpicture}[scale=0.5]
      \hspace{-5cm}
      \draw (0,-1) rectangle (1,0);
      \filldraw[fill=yellow] (0,0) rectangle (1,1);
      \draw (0,1) rectangle (1,2);
      \filldraw[fill=yellow] (0,2) rectangle (1,3);
      \draw (0,3) rectangle (1,4);
      \filldraw[fill=yellow] (0,4) rectangle (1,5);
      \draw (0,5) rectangle (1,6);
      \filldraw[fill=yellow] (0,6) rectangle (1,7);
      \filldraw[fill=yellow] (1,2) rectangle (2,3);
      \draw (1,3) rectangle (2,4);
      \filldraw[fill=yellow] (1,4) rectangle (2,5);
      \draw (1,5) rectangle (2,6);
      \filldraw[fill=yellow] (1,6) rectangle (2,7);
      \filldraw[fill=yellow] (2,4) rectangle (3,5);
      \draw (2,5) rectangle (3,6);
      \filldraw[fill=yellow] (2,6) rectangle (3,7);
      \filldraw[fill=yellow] (3,6) rectangle (4,7);
      \draw[densely dotted] (0,1)--(3,7);
      \draw[densely dotted] (0,-1)--(4,7);
      \draw[ultra thick,blue] (0,-1)--(1,-1)--(1,2)--(2,3)--(2,7)--(4,7);
      \draw[ultra thick,red] (0,-1)--(1,-1);
      \draw[ultra thick,red] (2,7)--(4,7);
      \draw (1,-3) node{(a) 4 flaws};

      \hspace{5cm}
      \draw (0,-1) rectangle (1,0);
      \filldraw[fill=yellow] (0,0) rectangle (1,1);
      \draw (0,1) rectangle (1,2);
      \filldraw[fill=yellow] (0,2) rectangle (1,3);
      \draw (0,3) rectangle (1,4);
      \filldraw[fill=yellow] (0,4) rectangle (1,5);
      \draw (0,5) rectangle (1,6);
      \filldraw[fill=yellow] (0,6) rectangle (1,7);
      \filldraw[fill=yellow] (1,2) rectangle (2,3);
      \draw (1,3) rectangle (2,4);
      \filldraw[fill=yellow] (1,4) rectangle (2,5);
      \draw (1,5) rectangle (2,6);
      \filldraw[fill=yellow] (1,6) rectangle (2,7);
      \filldraw[fill=yellow] (2,4) rectangle (3,5);
      \draw (2,5) rectangle (3,6);
      \filldraw[fill=yellow] (2,6) rectangle (3,7);
      \filldraw[fill=yellow] (3,6) rectangle (4,7);
      \draw[densely dotted] (0,1)--(3,7);
      \draw[densely dotted] (0,-1)--(4,7);
      \draw[ultra thick,blue] (0,-1)--(0,3)--(2,3)--(2,3)--(3,3)--(3,6)--(4,7);
      \draw[ultra thick,red] (0,3)--(2,3);
      \draw[ultra thick,red] (2,3)--(3,3);
      \draw (1,-3) node{(b) circular shifting};
      
      \hspace{5cm}
      \draw (0,-1) rectangle (1,0);
      \filldraw[fill=yellow] (0,0) rectangle (1,1);
      \draw (0,1) rectangle (1,2);
      \filldraw[fill=yellow] (0,2) rectangle (1,3);
      \draw (0,3) rectangle (1,4);
      \filldraw[fill=yellow] (0,4) rectangle (1,5);
      \draw (0,5) rectangle (1,6);
      \filldraw[fill=yellow] (0,6) rectangle (1,7);
      \filldraw[fill=yellow] (1,2) rectangle (2,3);
      \draw (1,3) rectangle (2,4);
      \filldraw[fill=yellow] (1,4) rectangle (2,5);
      \draw (1,5) rectangle (2,6);
      \filldraw[fill=yellow] (1,6) rectangle (2,7);
      \filldraw[fill=yellow] (2,4) rectangle (3,5);
      \draw (2,5) rectangle (3,6);
      \filldraw[fill=yellow] (2,6) rectangle (3,7);
      \filldraw[fill=yellow] (3,6) rectangle (4,7);
      \draw[densely dotted] (0,1)--(3,7);
      \draw[densely dotted] (0,-1)--(4,7);
      \draw[ultra thick,blue] (0,-1)--(0,2)--(2,2)--(2,3)--(3,3)--(3,6)--(4,7);
      \draw[ultra thick,red] (0,2)--(2,2);
      \draw[ultra thick,red] (2,3)--(3,3);
      \draw (1,-3) node{(c) 5 flaws};  
   \end{tikzpicture}
   \end{center}
   \vspace{-0.2cm}
   \caption{
   Rules for adding and subtracting flaws
   in Theorem~\ref{th-small-rule}(1c) and \ref{th-small-rule}(2c).}
   \label{small-rule-4}
   \end{figure}
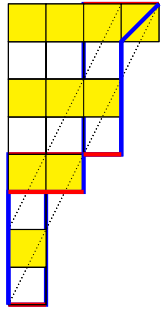        
      
         \item
         Similar to the previous case,
         $(-1,-2)$-circular shift is applied to the path in Figure~\ref{small-rule-5}(a) twice.
         Take the leftmost one ($2$-east run) and the rightmost one ($1$-east run) 
         of the newly obtained path in Figure~\ref{small-rule-5}(b), and
         move the $1$-east run to the position of the $2$-east run on the sequence.
         As the last step, 
         the $2$-east run is moved to the left one unit to get one more flawed north step 
         as shown in Figure~\ref{small-rule-5}(c).
   
         Conversely,
         the $2$-east run right before the leftmost flawed step 
         in Figure~\ref{small-rule-5}(c)
         is moved to the right one unit on the sequence and meets another east run.
         Then, move the right east run ($1$-east run) 
         between two east runs sharing the same horizontal line
         to the right until it meets a horizontal line $y=4$
         as in Figure~\ref{small-rule-5}(b). 
         Lastly, a $(2, 4)$-circular shift is applied to get a path 
         in Figure~\ref{small-rule-5}(a).
      \end{enumerate}  
            
   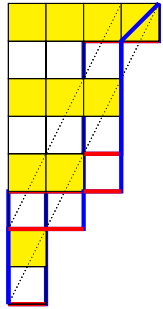
\begin{figure}[h]
   \begin{center}
   \begin{tikzpicture}[scale=0.5]
      \hspace{-5cm}
      \draw (0,-1) rectangle (1,0);
      \filldraw[fill=yellow] (0,0) rectangle (1,1);
      \draw (0,1) rectangle (1,2);
      \filldraw[fill=yellow] (0,2) rectangle (1,3);
      \draw (0,3) rectangle (1,4);
      \filldraw[fill=yellow] (0,4) rectangle (1,5);
      \draw (0,5) rectangle (1,6);
      \filldraw[fill=yellow] (0,6) rectangle (1,7);
      \filldraw[fill=yellow] (1,2) rectangle (2,3);
      \draw (1,3) rectangle (2,4);
      \filldraw[fill=yellow] (1,4) rectangle (2,5);
      \draw (1,5) rectangle (2,6);
      \filldraw[fill=yellow] (1,6) rectangle (2,7);
      \filldraw[fill=yellow] (2,4) rectangle (3,5);
      \draw (2,5) rectangle (3,6);
      \filldraw[fill=yellow] (2,6) rectangle (3,7);
      \filldraw[fill=yellow] (3,6) rectangle (4,7);
      \draw[densely dotted] (0,1)--(3,7);
      \draw[densely dotted] (0,-1)--(4,7);
      \draw[ultra thick,blue] (0,-1)--(1,-1)--(1,2)--(2,3)--(2,6)--(4,6)--(4,7);
      \draw[ultra thick,red] (0,-1)--(1,-1);
      \draw[ultra thick,red] (2,6)--(4,6);
      \draw (1,-3) node{(a) 5 flaws};
      
      \hspace{5cm}
      \draw (0,-1) rectangle (1,0);
      \filldraw[fill=yellow] (0,0) rectangle (1,1);
      \draw (0,1) rectangle (1,2);
      \filldraw[fill=yellow] (0,2) rectangle (1,3);
      \draw (0,3) rectangle (1,4);
      \filldraw[fill=yellow] (0,4) rectangle (1,5);
      \draw (0,5) rectangle (1,6);
      \filldraw[fill=yellow] (0,6) rectangle (1,7);
      \filldraw[fill=yellow] (1,2) rectangle (2,3);
      \draw (1,3) rectangle (2,4);
      \filldraw[fill=yellow] (1,4) rectangle (2,5);
      \draw (1,5) rectangle (2,6);
      \filldraw[fill=yellow] (1,6) rectangle (2,7);
      \filldraw[fill=yellow] (2,4) rectangle (3,5);
      \draw (2,5) rectangle (3,6);
      \filldraw[fill=yellow] (2,6) rectangle (3,7);
      \filldraw[fill=yellow] (3,6) rectangle (4,7);
      \draw[densely dotted] (0,1)--(3,7);
      \draw[densely dotted] (0,-1)--(4,7);
      \draw[ultra thick,blue] (0,-1)--(0,2)--(2,2)--(2,3)--(3,3)--(3,6)--(4,7);
      \draw[ultra thick,red] (0,2)--(2,2);
      \draw[ultra thick,red] (2,3)--(3,3);
      \draw (1,-3) node{(b) circular shifting};
      
      \hspace{5cm}
      \draw (0,-1) rectangle (1,0);
      \filldraw[fill=yellow] (0,0) rectangle (1,1);
      \draw (0,1) rectangle (1,2);
      \filldraw[fill=yellow] (0,2) rectangle (1,3);
      \draw (0,3) rectangle (1,4);
      \filldraw[fill=yellow] (0,4) rectangle (1,5);
      \draw (0,5) rectangle (1,6);
      \filldraw[fill=yellow] (0,6) rectangle (1,7);
      \filldraw[fill=yellow] (1,2) rectangle (2,3);
      \draw (1,3) rectangle (2,4);
      \filldraw[fill=yellow] (1,4) rectangle (2,5);
      \draw (1,5) rectangle (2,6);
      \filldraw[fill=yellow] (1,6) rectangle (2,7);
      \filldraw[fill=yellow] (2,4) rectangle (3,5);
      \draw (2,5) rectangle (3,6);
      \filldraw[fill=yellow] (2,6) rectangle (3,7);
      \filldraw[fill=yellow] (3,6) rectangle (4,7);
      \draw[densely dotted] (0,1)--(3,7);
      \draw[densely dotted] (0,-1)--(4,7);
      \draw[ultra thick,blue] (0,-1)--(0,1)--(2,1)--(2,2)--(3,2)--(3,6)--(4,7);
      \draw[ultra thick,red] (0,1)--(2,1);
      \draw[ultra thick,red] (2,2)--(3,2);
      \draw (1,-3) node{(c) 6 flaws};                         
   \end{tikzpicture}
   \end{center}
   \vspace{-0.2cm}
   \caption{
   Rules for adding and subtracting flaws
   in Theorem~\ref{th-small-rule}(1d) and \ref{th-small-rule}(2d).}
   \label{small-rule-5}
   \end{figure}  
   \end{example}  
   
\section{Fuss-\Sch paths and noncrossing partitions}

   In this section, 
   to extend the results of small $(k, r)$-Fuss-Schröder paths with type $\lambda$
   to large $(k, r)$-Fuss-Schröder paths with type $\lambda$,
   we introduce sparse noncrossing partitions
   which are in bijection with the set of $(k, r)$-Fuss-Schröder paths of type $\lambda$.

   A \emph{noncrossing partition} of $[n] := \{1, 2, \ldots, n\}$ is 
   a pairwise disjoint subsets $B_1, B_2. \ldots, B_l$ of $[n]$ whose union is $[n]$ in which, 
   if $a$ and $b$ belong to one block~$B_i$ and $x$ and $y$ to another block $B_j$,
   they are not arranged in the order $a x b y$.
   Note that, if the elements $1, 2, \dots, n$ are equally-spaced dots on a horizontal line, and
   all the successive elements of the same bock are connected by arc above the line,  
   then no arches cross each other for a noncrossing partition of $[n]$.
   A noncrossing partition is called \emph{sparse} 
   if no two consecutive integers are in the same block.
   We give an order to the blocks by the order of the smallest element of each block.
   \emph{Connected components} of a noncrossing partition of $[n]$ are 
   $\{1, 2, \ldots, i_1\}, \{i_1+1, i_1+2, \ldots, i_2\}, \ldots, \{i_t+1, i_t+2, \ldots, n\}$
   where $i_1$ is the greatest element of the block containing $1$, 
   $i_2$ is the greatest element of the block containing $i_1+1$, and so on. 
   The \emph{arc type} of a noncrossing partition is the integer partition
   obtained from the numbers of connected arcs.
   
   In Figure~\ref{ncpartition}, 
   a noncrossing partition of $[10]$ is written by $4$ ordered blocks. 
   Two connected components are $\{1, 2, \ldots, 8\}$ and $\{9, 10\}$, and 
   the arc type of the given noncrossing partition is $(3,2,1)$.
   Since $2$ and $3$ are in the same block,
   this partition is not a sparse noncrossing partition.
   
   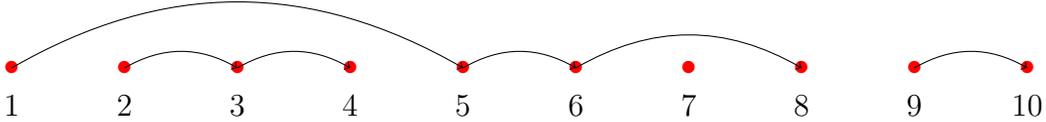
\begin{figure}[h]
   \begin{center}
   \begin{tikzpicture}[scale=0.5]
      \hspace{0cm}
      \draw[red] (0,1) node{\textbullet};
      \draw[red] (3,1) node{\textbullet};
      \draw[red] (6,1) node{\textbullet};
      \draw[red] (9,1) node{\textbullet};
      \draw[red] (12,1) node{\textbullet};
      \draw[red] (15,1) node{\textbullet};
      \draw[red] (18,1) node{\textbullet};
      \draw[red] (21,1) node{\textbullet};
      \draw[red] (24,1) node{\textbullet};
      \draw[red] (27,1) node{\textbullet};
      \draw (0,0) node{$1$};
      \draw (3,0) node{$2$};
      \draw (6,0) node{$3$};
      \draw (9,0) node{$4$};
      \draw (12,0) node{$5$};
      \draw (15,0) node{$6$};
      \draw (18,0) node{$7$};
      \draw (21,0) node{$8$};
      \draw (24,0) node{$9$};
      \draw (27,0) node{$10$};       
      \draw[->] (0,1) to[bend left] (12,1);
      \draw[->] (12,1) to[bend left] (15,1);
      \draw[->] (15,1) to[bend left] (21,1);
      \draw[->] (3,1) to[bend left] (6,1);
      \draw[->] (6,1) to[bend left] (9,1); 
      \draw[->] (24,1) to[bend left] (27,1);  
   \end{tikzpicture}
   \end{center}
   \vspace{-0.2cm}
   \caption{Noncrossing partition $(\{1,5,6,8\}$, $\{2,3,4\}$, $\{7\}$, $\{9,10\})$ of $[10]$.}
   \label{ncpartition}
   \end{figure}

   Before introducing a special noncrossing partition,
   we need more notations and labellings concerning $(k, k)$-Fuss-\Sch paths. 
   On a skew shape from $(0,0)$ to $(n,kn)$, 
   the $(n-i)k$th row is labeled by $i(k+1)+2$ for $i \geq 0$, and
   the horizontal line $y=(n-i)k+j$ is labeled by $i(k+1)+1-j$ if $0 \leq j < k$.   
   See Figure~\ref{labels} for example.
   
   \begin{figure}[h]
   \begin{center}
   \begin{tikzpicture}[scale=0.5]
      \hspace{0cm}      
      \draw (0,-1) rectangle (1,0);
      \filldraw[fill=yellow] (0,0) rectangle (1,1);
      \draw (0,1) rectangle (1,2);
      \filldraw[fill=yellow] (0,2) rectangle (1,3);
      \draw (0,3) rectangle (1,4);
      \filldraw[fill=yellow] (0,4) rectangle (1,5);
      \draw (0,5) rectangle (1,6);
      \filldraw[fill=yellow] (0,6) rectangle (1,7);
      \filldraw[fill=yellow] (1,2) rectangle (2,3);
      \draw (1,3) rectangle (2,4);
      \filldraw[fill=yellow] (1,4) rectangle (2,5);
      \draw (1,5) rectangle (2,6);
      \filldraw[fill=yellow] (1,6) rectangle (2,7);
      \filldraw[fill=yellow] (2,4) rectangle (3,5);
      \draw (2,5) rectangle (3,6);
      \filldraw[fill=yellow] (2,6) rectangle (3,7);
      \filldraw[fill=yellow] (3,6) rectangle (4,7);
      \draw[densely dotted] (0,1)--(3,7);
      \draw[densely dotted] (0,-1)--(4,7);  
      \draw[red] (0,-1) node{\textbullet};  
      \draw (-1.2,-1.2) node{$(0,0)$};
      \hspace{0.2cm}
      \draw (2,-1) node{$\longleftarrow 13$};
      \draw (-1.8,0.5) node{$11 \longrightarrow$};
      \draw (2,0) node{$\longleftarrow 12$};
      \draw (2,1) node{$\longleftarrow 10$};
      \draw (-1.8,2.5) node{$~8 \longrightarrow$};
      \draw (3,2) node{$\longleftarrow 9~$};
      \draw (3,3) node{$\longleftarrow 	7~$};
      \draw (-1.8,4.5) node{$~5 \longrightarrow$};
      \draw (4,4) node{$\longleftarrow 6~$};
      \draw (4,5) node{$\longleftarrow 4~$};
      \draw (-1.8,6.5) node{$~2 \longrightarrow$};
      \draw (5,6) node{$\longleftarrow 3~$};
      \draw (5,7) node{$\longleftarrow 1~$};
      \draw[ultra thick,blue] (-0.35,-1)--(-0.35,5)--(0.65,5)--(0.65,6)--(1.65,7)--(3.65,7);
   \end{tikzpicture}
   \end{center}
   \vspace{-0.2cm}
   \caption{Labels when $n=4$ and $k=2$.}
   \label{labels}
   \end{figure}
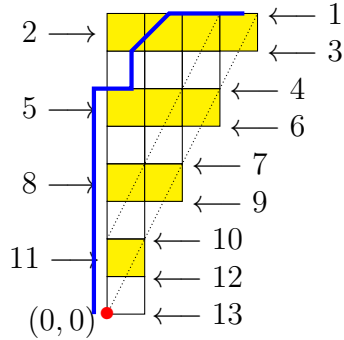 
   
   We can represent a $(k, k)$-Fuss-\Sch path of length $n$ 
   as a sequence $s_1 \leq s_2 \leq \cdots \leq s_n$
   such that $s_i$'s are the labels in which $E$ and $D$ steps located.
   In Figure~\ref{labels}, 
   the sequence corresponding to the path is $1124$.
   Then, we trace sequence representations of $(k, k)$-Fuss-\Sch path as follows:
   \begin{enumerate}
      \item Start with two numbers $1$ and $2$ on a horizontal line, 
      and read a sequence from $s_{1}$ to $s_{n}$.
      \item For $m$ consecutive east steps labeled by $s_j = s_{j+1} = \cdots s_{j+m-1}$,
      \begin{enumerate}
         \item replace each number $i(> s_j)$ with $i + m(k + 1)$, and
         \item replace $s_{j}$ with 
               $s_j, s_j + 1, s_j, s_j + 2, s_j, . . . , s_j, s_j + m(k + 1), s_j$.
      \end{enumerate}   
      \item For a diagonal step labeled by $s_j$,
      \begin{enumerate}
         \item replace each number $i(> s_j)$ with $i + k + 1$, and 
         \item replace $s_{j}$ with 
               $s_j, s_j + 1, s_j, s_j + 2, s_j, . . . , s_j, s_j + (k + 1), s_j$.
      \end{enumerate}
      \item For a north step, do nothing. 
   \end{enumerate}   
   Then the resulting sequence gives a sparse noncrossing partition of [2(k+1)n + 2]
   in which the elements at the number $i$ positions are the elements of $i$th block.
   See Example~\ref{ncexample}. 
   
   \begin{example}
   \label{ncexample}
   The sparse noncrossing partition corresponding to 
   a $(k, k)$-Fuss-\Sch path in Figure~\ref{labels}
   is given in Figure~\ref{ncbijection}.
   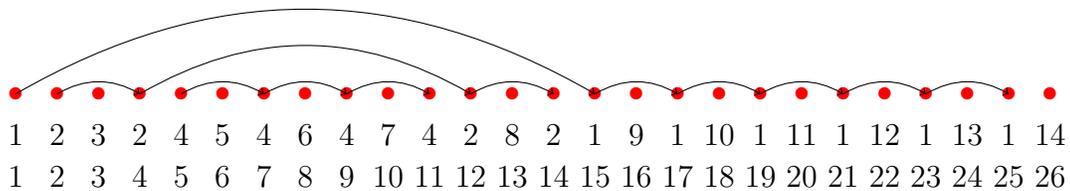
\begin{figure}[h]
   \begin{center}
   \begin{tikzpicture}[scale=0.55]
      \hspace{0cm}
      \draw[red] (0,1) node{\textbullet};
      \draw[red] (1,1) node{\textbullet};
      \draw[red] (2,1) node{\textbullet};
      \draw[red] (3,1) node{\textbullet};
      \draw[red] (4,1) node{\textbullet};
      \draw[red] (5,1) node{\textbullet};
      \draw[red] (6,1) node{\textbullet};
      \draw[red] (7,1) node{\textbullet};
      \draw[red] (8,1) node{\textbullet};
      \draw[red] (9,1) node{\textbullet};
      \draw[red] (10,1) node{\textbullet};
      \draw[red] (11,1) node{\textbullet};
      \draw[red] (12,1) node{\textbullet};
      \draw[red] (13,1) node{\textbullet};
      \draw[red] (14,1) node{\textbullet};
      \draw[red] (15,1) node{\textbullet};
      \draw[red] (16,1) node{\textbullet};
      \draw[red] (17,1) node{\textbullet};
      \draw[red] (18,1) node{\textbullet};
      \draw[red] (19,1) node{\textbullet};
      \draw[red] (20,1) node{\textbullet};
      \draw[red] (21,1) node{\textbullet};
      \draw[red] (22,1) node{\textbullet};
      \draw[red] (23,1) node{\textbullet};
      \draw[red] (24,1) node{\textbullet};
      \draw[red] (25,1) node{\textbullet};
      \draw (0,0) node{$1$};
      \draw (1,0) node{$2$};
      \draw (2,0) node{$3$};
      \draw (3,0) node{$2$};
      \draw (4,0) node{$4$};
      \draw (5,0) node{$5$};
      \draw (6,0) node{$4$};
      \draw (7,0) node{$6$};
      \draw (8,0) node{$4$};
      \draw (9,0) node{$7$};
      \draw (10,0) node{$4$};
      \draw (11,0) node{$2$};
      \draw (12,0) node{$8$};
      \draw (13,0) node{$2$};
      \draw (14,0) node{$1$};
      \draw (15,0) node{$9$};
      \draw (16,0) node{$1$};
      \draw (17,0) node{$10$};
      \draw (18,0) node{$1$};
      \draw (19,0) node{$11$}; 
      \draw (20,0) node{$1$};
      \draw (21,0) node{$12$};
      \draw (22,0) node{$1$};
      \draw (23,0) node{$13$};
      \draw (24,0) node{$1$};
      \draw (25,0) node{$14$};
      \draw (0,-1) node{$1$};
      \draw (1,-1) node{$2$};
      \draw (2,-1) node{$3$};
      \draw (3,-1) node{$4$};
      \draw (4,-1) node{$5$};
      \draw (5,-1) node{$6$};
      \draw (6,-1) node{$7$};
      \draw (7,-1) node{$8$};
      \draw (8,-1) node{$9$};
      \draw (9,-1) node{$10$};
      \draw (10,-1) node{$11$};
      \draw (11,-1) node{$12$};
      \draw (12,-1) node{$13$};
      \draw (13,-1) node{$14$};
      \draw (14,-1) node{$15$};
      \draw (15,-1) node{$16$};
      \draw (16,-1) node{$17$};
      \draw (17,-1) node{$18$};
      \draw (18,-1) node{$19$};
      \draw (19,-1) node{$20$}; 
      \draw (20,-1) node{$21$};
      \draw (21,-1) node{$22$};
      \draw (22,-1) node{$23$};
      \draw (23,-1) node{$24$};
      \draw (24,-1) node{$25$};
      \draw (25,-1) node{$26$};      
      \draw[->] (4,1) to[bend left] (6,1);
      \draw[->] (6,1) to[bend left] (8,1);
      \draw[->] (8,1) to[bend left] (10,1);
      \draw[->] (1,1) to[bend left] (3,1);
      \draw[->] (3,1) to[bend left] (11,1); 
      \draw[->] (11,1) to[bend left] (13,1);  
      \draw[->] (0,1) to[bend left] (14,1);
      \draw[->] (14,1) to[bend left] (16,1);
      \draw[->] (16,1) to[bend left] (18,1);
      \draw[->] (18,1) to[bend left] (20,1);
      \draw[->] (20,1) to[bend left] (22,1); 
      \draw[->] (22,1) to[bend left] (24,1); 
   \end{tikzpicture}
   \end{center}
   \vspace{-0.2cm}
   \caption{Sparse noncrossing partition of $[26]$ with $14$ blocks,
   $(\{1, 15, 17, 19, 21, 23, 25\}$, $\{2, 4, 12, 14\}$, $\{3\}$, $\{5, 7, 9, 11\}$, $\{6\}$,
   $\{8\}$, $\{10\}$, $\{13\}$, $\{16\}$, $\{18\}$, $\{20\}$, $\{22\}$, $\{24\}$, $\{26\})$.}
   \label{ncbijection}
   \end{figure}    
   \end{example}
   
   See the following conjecture about
   characteristics of sparse noncrossing partitions from $(k, k)$-Fuss-\Sch paths.

   \begin{conjecture}      
      The set of large $(k, k)$-Fuss-\Sch paths of 
      type $\lambda = \lambda_1 \lambda_2 \ldots \lambda_\ell$ of length $n$
      is in bijection with 
      the set of sparse noncrossing partitions of $[2(k+1)n+2]$ 
      with $2$ connected components such that
      \begin{enumerate}
         \item
         the arc type is
         $((k+1)\lambda_1, (k+1)\lambda_2, \dots, (k+1)\lambda_\ell, (k+1)^{n-\abs{\lambda}})$,
         \item
         the set of $(i(k+1)+2)$th blocks consists of 
         $n-\abs{\lambda}$ blocks of arc type $k+1$ and $\abs{\lambda}$ singleton blocks
         where $0 \leq i \leq n-1$, and  
         \item
         the set of last $t(k+1)$ blocks has at least $t(k-1)+1$ singleton blocks for $t \geq 1$.
      \end{enumerate}
   \end{conjecture}

   Note that, if the second connected component is a singleton block,
   the corresponding $(k, k)$-Fuss-\Sch path is a small $(k, k)$-Fuss-\Sch path.
   Hence, we have a next conjecture. 
   
   \begin{conjecture} 
      The set of small $(k, k)$-Fuss-\Sch paths of 
      type $\lambda = \lambda_1 \lambda_2 \ldots \lambda_\ell$ of length $n$
      is in bijection with 
      the set of connected sparse noncrossing partitions of $[2(k+1)n+1]$ such that
      \begin{enumerate}
         \item
         the arc type is 
         $((k+1)\lambda_1, (k+1)\lambda_2, \dots, (k+1)\lambda_\ell, (k+1)^{n-\abs{\lambda}})$,
         \item
         the set of $(i(k+1)+2)$th blocks consists of 
         $n-\abs{\lambda}$ blocks of arc type $k+1$ and $\abs{\lambda}$ singleton blocks
         where $0 \leq i \leq n-1$, and 
         \item
         the set of last $t(k+1)$ blocks has at least $t(k-1)+1$ singleton blocks for $t \geq 1$. 
      \end{enumerate}
   \end{conjecture}

   However, it has a different correspondence for the case that 
   the first connected component is a singleton block.
   In this case, all the partitions counts 
   $(k, k)$-Fuss-\Sch paths starting with a diagonal step.

\bibliographystyle{plain}

\begin{thebibliography}{1}

\bibitem{AnEuKim}
Su~Hyung An, Sen-Peng Eu, and Sangwook Kim.
\newblock Large {S}chr\"oder paths by types and symmetric functions.
\newblock {\em Bull. Korean Math. Soc.}, 51(4):1229--1240, 2014.

\bibitem{Armstrong}
Drew Armstrong.
\newblock Generalized noncrossing partitions and combinatorics of {C}oxeter
  groups.
\newblock {\em Mem. Amer. Math. Soc.}, 202(949):x+159, 2009.

\bibitem{EuFu}
Sen-Peng Eu and Tung-Shan Fu.
\newblock Lattice paths and generalized cluster complexes.
\newblock {\em J. Combin. Theory Ser. A}, 115(7):1183--1210, 2008.

\bibitem{Kreweras}
G.~Kreweras.
\newblock Sur les partitions non crois\'ees d'un cycle.
\newblock {\em Discrete Math.}, 1(4):333--350, 1972.

\bibitem{Liaw}
S.~C. Liaw, H.~G. Yeh, F.~K. Hwang, and G.~J. Chang.
\newblock A simple and direct derivation for the number of noncrossing
  partitions.
\newblock {\em Proc. Amer. Math. Soc.}, 126(6):1579--1581, 1998.

\bibitem{ParkKim}
Youngja Park and Sangwook Kim.
\newblock Chung-{F}eller property of {S}chr\"oder objects.
\newblock {\em Electron. J. Combin.}, 23(2):Paper 2.34, 14, 2016.

\end{thebibliography}

\def\cprime{$'$}

\end{document}